\documentclass[11pt]{revtex4}

\usepackage{graphicx,epsfig,amssymb,amsbsy,amsmath,amsthm}
\newcommand{\bb}[1]{\left({#1}\right)}					
\newcommand{\sq}[1]{\left[#1\right]}						
\newcommand{\cc}[1]{\left\{#1\right\}}					
\newcommand{\op}[1]{\mathcal{#1}}
\newcommand{\ord}[1]{{\sf O}\bb{#1}}					

\newcommand{\sfrac}[2]{\mbox{$\frac{#1}{#2}$}}	

\newcommand{\hf}{\mbox{$\frac12$}}

\newcommand{\e}{e }

\newtheorem{theorem}{Theorem}
\newtheorem{lemma}[theorem]{Lemma}

\newtheorem{proposition}[theorem]{Proposition}

\newcommand{\sign}{\operatorname{sign}}

\newcommand{\fref}[1]{figure~\ref{#1}}
\newcommand{\Fref}[1]{Figure~\ref{#1}}

\newcommand{\eref}[1]{(\ref{#1})}

\newcommand{\sref}[1]{section~\ref{#1}}

\def\eps{\varepsilon}

\begin{document}

\title{Fold singularities of nonsmooth and slow-fast dynamical systems -- equivalence through regularization}
\author{Mike R. Jeffrey}\address{Dept. of Engineering Mathematics, University of Bristol, Merchant Venturer's Building, Bristol BS8 1UB, UK, email: mike.jeffrey@bristol.ac.uk}
\date{\today}

\begin{abstract} 
The {\it two-fold singularity} has played a significant role in our understanding of uniqueness and stability in piecewise smooth dynamical systems. When a vector field is discontinuous at some hypersurface, it can become tangent to that surface from one side or the other, and tangency from both sides creates a two-fold singularity. The local flow bears a superficial resemblance to so-called {\it folded singularities} in (smooth) slow-fast systems, which arise at the intersection of attractive and repelling branches of slow invariant manifolds, important in the local study of canards and mixed mode oscillations. Here we show that these two singularities are intimately related. When the discontinuity in a piecewise smooth system is regularized (smoothed out) at a two-fold singularity, the resulting system can be mapped onto a folded singularity. The result is not obvious since it requires the presence of nonlinear or `hidden' terms at the discontinuity, which turn out to be necessary for structural stability of the regularization (or smoothing) of the discontinuity, and necessary for mapping to the folded singularity. 
\end{abstract}

\maketitle

\section{Introduction}\label{sec:intro}

If a flow is piecewise smooth, having a vector field that is discontinuous on some hypersurface $\Sigma$, then under generic conditions there can exist isolated singularities where the flow curves (or `folds' parabolically) towards or away from $\Sigma$ on both sides of the surface. The result is a two-fold singularity, as depicted in \fref{fig:2folddet}, generic in systems of three or more dimensions. 
\begin{figure}[h!]\centering\includegraphics[width=\textwidth]{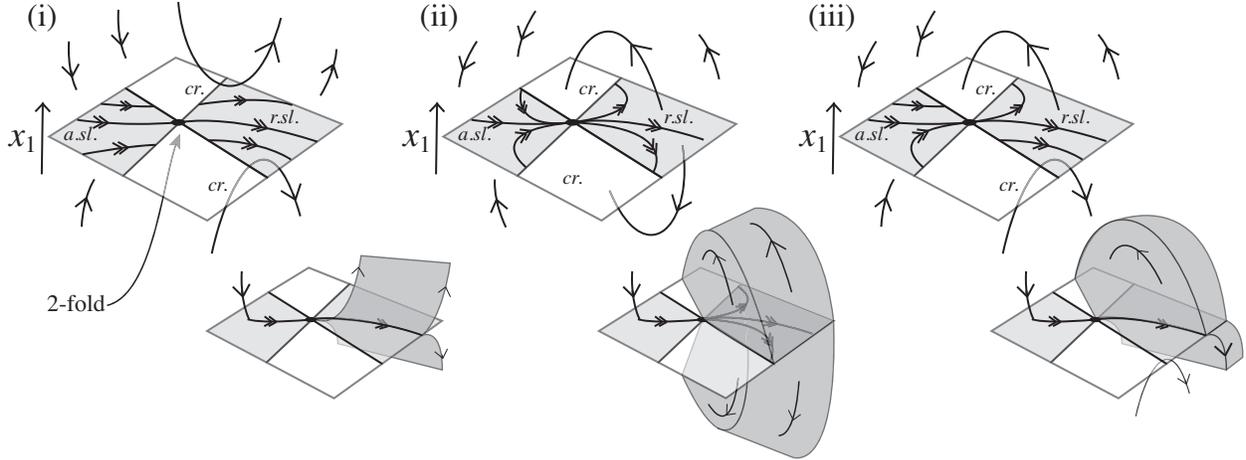}
\vspace{-0.3cm}\caption{\small\sf Three kinds of two-fold. The main figure shows the phase portrait: in the unshaded regions the flow crosses ($cr.$) through a discontinuity at $x_1=0$, in the shaded regions the flow can only {\it slide} along the discontinuity on $x_1=0$, the region being attracting ($a.sl.$) or repelling ($r.sl.$). In the examples shown, determinacy-breaking occurs at the singularity, meaning that the flow there becomes set-valued, the set has 2 dimensions in (i) and 3 dimensions in (ii-iii). }\label{fig:2folddet}\end{figure}
%

Consider a piecewise smooth system of the form
\begin{equation}\label{fns}
\dot {\bf x}
=\left\{\begin{array}{lll}{\bf f}({\bf x};+1)&\rm if&x_1>0\;,\\{\bf f}({\bf x};-1)&\rm if&x_1<0\;,\end{array}\right.
\end{equation}
where $\bf f$ is differentiable with respect to the variable ${\bf x}=(x_1,x_2,x_3)$
A two-fold singularity is a point ${\bf x}_p=(x_{1p},x_{2p},x_{3p})$ at which  
\begin{equation}\label{2fdef}
0=x_1=\lim_{\delta\rightarrow0}\left.\dot x_1\right|_{x_1=\pm\delta}\neq\lim_{\delta\rightarrow0}\left.\ddot x_1\right|_{x_1=\pm\delta}
\end{equation}
and for which the two curves defined by $\lim_{\delta\rightarrow0}\left.\dot x_1\right|_{x_1=\pm\delta}=0$ on $x_1=0$ are transversal. A more complete description can be found in \cite{jc12}. 

For a local singularity that is so easy to define, the `two-fold' singularity has proven surprisingly difficult to characterize, from its first description in \cite{f88,t90} in three dimensions, to its study in higher dimensions in \cite{jc12}. These exclusively consider the class of {\it Filippov} system obtained by expressing \eref{fns} as the convex combination
\begin{equation}\label{ffil}
\dot{\bf x}={\bf f}({\bf x};\lambda)=\frac{1+\lambda}2{\bf f}({\bf x};+1)+\frac{1-\lambda}2{\bf f}({\bf x};-1)\;,
\end{equation}
where 
\begin{equation}\label{sign}
\lambda\in\left\{\begin{array}{lll}\sign\bb{x_1}&\rm if&x_1\neq0\;,\\\sq{-1,+1}&\rm if&x_1=0\;.\end{array}\right.
\end{equation}
The dynamics on the discontinuity surfaces in \fref{fig:2folddet} arise, for example, from this expression. 
Early questions about the structural stability of two-fold singularities in such systems have been resolved by uncovering their intricate phase portraits, revealing various topologically stable phase portraits separated by bifurcations \cite{jc09,jc12,j12xing}, which include the birth of limit cycles, bifurcation of an invariant nonsmooth diabolo, and passage of equilibria through the singularity. In its structurally stable forms (i.e. away from any bifurcations), the two-fold is neither an attractor nor repellor, so the flow either misses the singularity, or traverses it in finite time. 

It is in the latter that important questions remain concerning the two-fold singularity, particularly in cases where it forms a bridge {\it from} an attracting region on the discontinuity surface {\it into} a repelling region, mimicking {\it canard} behaviour of smooth two-timescale systems \cite{b81}. In a loose definition suitable to both smooth and nonsmooth dynamics, canards are trajectories that persist from an attracting invariant manifold to a repelling invariant manifold. Numerous canards can be seen in \fref{fig:2folddet}. Such behaviour takes a more extreme form in nonsmooth systems, because the two-fold breaks determinacy in both forward and backward time through the singularity. This is illustrated in \fref{fig:2folddet}, where a typical single trajectory is shown entering the singularity, being deterministic until it does so, and afterwards exploding into a set-valued flow of infinite onward trajectories. This shape of this outset of the flow is determined by the local vector fields. 


Particularly because of some similarity to canard dynamics, attention has turned to how the two-fold can be understood as a limit or approximation of a smooth flow. An equivalence between ``sliding'' motion along a discontinuity surface, and ``slow'' motion on invariant manifolds of a smooth two-timescale system, has been shown \cite{st96,j15douglas}. Concerning two-folds, the relation between the sliding phase portraits at two-fold singularities, and canard dynamics of two-timescale systems, have received growing attention \cite{t07,tls08,jd10,jd12,hk15}. In \cite{jd10,jd12}, a qualitative association was made with the so-called `folded' singularities of two timescale systems. 

A folded singularity can be defined in a system
\begin{equation}\label{fsf}
\bb{\eps\dot x_1,\dot x_2,\dot x_3}=\bb{g_1(x_1,x_2,x_3;\eps),g_2(x_1,x_2,x_3;\eps),g_3(x_1,x_2,x_3;\eps)}\;,
\end{equation}
where the $g_i$ are differentiable, with $0<\eps\ll1$, as a point satisfying
\begin{equation}\label{fnodedef}
0=g_1=\dot g_1=\frac{\partial g_1}{\partial x_1}\;,\qquad\frac{\partial g_1}{\partial x_{2,3}}\neq0\neq\frac{\partial^2 g_1}{\partial x_1^2}\;.
\end{equation}
We shall prove here a more direct connection between the two singularities, by showing equivalence between the singularities defined by \eref{2fdef} and \eref{fnodedef} under explicit coordinate transformations. 

This approach offers a different viewpoint on desingularizing the two-fold singularity to that taken in \cite{t07,tls08,hk15}. There, blow up methods are used to show that a regularization (a smoothing) of the two-fold contains canards. The results there apply specifically to the Sotomayor-Teixeira regularization \cite{st96}, essentially replacing the sign function $\lambda$ in the convex combination \eref{ffil}, with a smooth sigmoid function of $x_1$. While not currently in common use, it is easy to show that more general forms of system are possible, replacing \eref{ffil} with the piecewise smooth systems 
\begin{equation}\label{fnon}
\dot{\bf x}={\bf f}({\bf x};\lambda)=\frac{1+\lambda}2{\bf f}({\bf x};+1)+\frac{1-\lambda}2{\bf f}({\bf x};-1)+(1-\lambda^2){\bf g}({\bf x};\lambda)\;,
\end{equation}
using \eref{sign}, which like \eref{ffil} coincides with \eref{fns} for $x_1\neq0$. The function $\bf g$ is an arbitrary finite vector field, and the possibility of nonlinear $\lambda$ dependence can be found discussed already in \cite{f88,u92,seidman96} (including an experimental model in \cite{seidman96}). Sometimes called `hidden' terms because they vanish everywhere except at the discontinuity, a general approach for handling nonlinear dependence on $\lambda$ was introduced conceptually in \cite{j13error}. The Sotomayor-Teixeira theory of regularization was extended to \eref{fnon} in \cite{j15douglas}. This provides regularized systems that are not included in the Sotomayor-Teixeira-Filippov approach through \eref{ffil}, an issue we will illustrate with a simple example in \sref{sec:prelude}. Such systems turn out to be essential for the equivalence we seek to prove here. 


In \sref{sec:2fold} we introduce the normal form of the two-fold singularity, and outline the basic steps for its study by regularizing the discontinuity in \sref{sec:blowup}. 
In \sref{sec:2fdummy} we regularize the normal piecewise smooth system, assuming only linear dependence on $\lambda$ as in the standard literature, but show that this results in a degenerate system. 
In \sref{sec:2fpert} we perturb this using nonlinear dependence on $\lambda$, finding that it breaks the degeneracyy, and can be mapped onto the folded singularity of a smooth two timescale system. Remarks showing that these results follow also if we blow up, rather than regularize, the discontinuity, are given in \sref{sec:prelude}.



\section{The two-fold singularity}\label{sec:2fold}

The normal form of the two-fold singularity is 
\begin{equation}\label{2f}
\bb{\dot x_1,\dot x_2,\dot x_3}=
\left\{\begin{array}{lll}\bb{-x_2,a_1,b_1}&\rm if&x_1>0\;,\\\bb{+x_3,b_2,a_2}&\rm if&x_1<0\;,\end{array}\right.
\end{equation}
in terms of constants $a_i=\pm1$ and $b_i\in\mathbb R$. By results in \cite{f88,t93,jc12}, a system is locally approximated by \eref{2f} when it satisfies the conditions in \eref{2fdef}. 
The brief outline of the dynamics that follows is not essential to the following sections, but we give it for completeness. 

The local flow `folds' towards or away from the switching surface $x_1=0$, along the line $x_2=x_1=0$ on one side of the surface, and along the line $x_3=x_1=0$ on the other. Hence the point where these lines cross is called the `{\it two}-fold'. As a result, the surface $x_1=0$ is attractive in $x_2,x_3>0$ and repulsive in $x_2,x_3<0$, while trajectories cross the surface transversely in $x_2x_3<0$. In the attractive and repulsive regions the flow {\it slides} along the surface $x_1=0$, and follows a vector field that is found by substituting \eref{2f} into \eref{ffil}, and solving for $\lambda$ such that $\dot x_1=0$. We will not discuss this sliding dynamics in detail, see for example \cite{jc12} and references therein. 

The qualitative picture is then as shown in \fref{fig:2folddet}. 
The precise form of the local dynamics depends on whether the flow curves towards or away from the discontinuity, determined by $a_1$ and $a_2$, and also depends crucially on the quantity $b_1b_2$, which quantifies the jump  in the angle of the flow across the discontinuity. An accounting of the many classes of dynamics that arise from these simple conditions is given in \cite{jc12}, we give only the pertinent details here.

The three main `flavours' of two-fold are: the visible two-fold for $a_1=a_2=-1$, the invisible two-fold for $a_1=a_2=1$, and the mixed two-fold for $a_1a_2=-1$; an example of each is shown in \fref{fig:2folddet} (i,ii,iii) respectively. The terms {\it visible} or {\it invisible} indicate that the flow is curving away from or towards the discontinuity surface, respectively. In the cases depicted, there exist one or more {\it canard} trajectories, passing from the attractive sliding region to the repelling sliding region. This passage occurs in finite time (since the vector field obtained by substituting \eref{2f} into \eref{ffil} is non-vanishing everywhere locally). The flow is unique in forward time everywhere except in the repelling sliding region, where it is set-valued because trajectories may slide along $x_1=0$, but may also be ejected into $x_1>0$ or $x_1<0$ at any point. This means that the flow may evolve deterministically until it arrives at the singularity by means of a canard, at which point it becomes set-valued, so we say that {determinacy breaking} occurs at the singularity whenever canards are exhibited. This occurs in the invisible case when $b_1,b_2<0$ and $b_1b_2>1$, in the visible case when $b_1<0$ or $b_2<0$ or $b_1b_2<1$, and finally in the mixed case when $b_1<0<b_2$ and $b_1b_2<-1$ or when $b_1+b_2<0$ and $b_1-b_2<-2$.
(The particular cases shown in \fref{fig:2folddet} are: (i) $a_1=a_2=-1$ with $b_1<0$ or $b_2<0$ or $b_1b_2<1$; (ii) $a_1=a_2=1$ with $b_1,b_2<0$ and $b_1b_2>1$; (iii) $a_1a_2=-1$ with $b_1<0<b_2$ and $b_1b_2<-1$ or with $b_1+b_2<0$ and $b_1-b_2<-2$.)

This brief review follows the conventional picture, obtained by substituting \eref{2f} into \eref{ffil}. We shall only make the more general substitution of \eref{2f} into \eref{fnon} when necessary, and will consider $\bf g$ to be a small perturbation, which will therefore have only a small effect on the dynamics outlined above; this can be studied more closely but is not our main concern here. Our interest now lies in what happens when we smooth out this system. 


%
%
%

\section{Regularizing the piecewise smooth system}\label{sec:blowup}

We shall first outline the basics of regularization in a general form. 
Let ${\bf x}=(x_1,x_2,x_3)$ and ${\bf f}=(f_1,f_2,f_3)$. 
We regularize the vector field in \eref{fnon} by replacing $\lambda$ with a smooth sigmoid function 
\begin{equation}\label{Peps}
\phi_\eps(x_1)\in\left\{\begin{array}{lll}\sign(x_1)&\rm if&|x_1|>\eps\\\sq{-1,+1}&\rm if&|x_1|\le\eps\end{array}\right\}+\ord{\eps}\;,\qquad \phi_\eps'(x_1)>0\;\;\mbox{for}\;|x_1|<\eps\;,
\end{equation}
for some small positive parameter $\eps$. In the Sotomayor-Teixeira regularization it is usually assumed more strongly that $\phi_\eps(x_1)={\rm sign}(v)$ for $|x_1|>\eps$, without the $\ord\eps$ term, but this not crucial to our results. 
It is useful to introduce a fast variable $u=x_1/\eps$. Noting by \eref{Peps} that $\phi_\eps(\eps u)=\phi_1(u)$, we obtain
\begin{equation}\label{blowup}
(\eps\dot u,\dot x_2,\dot x_3)=\bb{f_1(\eps u,x_2,x_3;\phi_1(u),\;f_2(\eps u,x_2,x_3;\phi_1(u)),\;f_3(\eps u,x_2,x_3;\phi_1(u))}\;.
\end{equation}

This is the regularization of \eref{fnon}. 
The region $|x_1|\le\eps$, which collapses to the discontinuity surface $x_1=0$ in the limit $\eps\rightarrow0$, has been rescaled into the region $|u|\le1$. In the following we restrict our attention to the dynamics in the region $|u|<1$. 

Some basic properties of the system \eref{blowup} then follow using standard concepts of geometric singular perturbation theory (more detail of which can be found in \cite{f79,j95}, or various recent works in slow-fast dynamics such as \cite{w05} which will be relevant later). 

Expressing \eref{blowup} system with respect to a fast time $\tau=t/\eps$, denoting the corresponding time derivative with a prime, gives the fast subsystem
\begin{equation}\label{blowoute}
(u',x_2',x_3')=\bb{f_1(\eps u,x_2,x_3;\phi_1(u)),\;\eps f_2(\eps u,x_2,x_3;\phi_1(u)),\;\eps f_3(\eps u,x_2,x_3;\phi_1(u))}\;,
\end{equation} 
Setting $\eps=0$ gives the critical fast subsystem (sometimes called the {\it layer} problem)
\begin{equation}\label{blowout}
(u',x_2',x_3')=\bb{f_1(0,x_2,x_3;\phi_1(u)),\;0,\;0}\;.
\end{equation}
This is a one dimensional subsystem, with a set of $(x_2,x_3)$-parameterized equilibria inhabiting a surface
\begin{equation}\label{Mdef}
\op M^S=\cc{(u,x_2,x_3)\in\mathbb R^3:\;f_1(0,x_2,x_3;\phi_1(u))=0\;,|u|<1}\;.
\end{equation}
This is an invariant manifold of \eref{blowout} wherever $\op M^S$ is normally hyperbolic, thus excluding the set of points where hyperbolicity fails, defined as
\begin{equation}\label{Ldef}
\op L=\cc{(u,x_2,x_3)\in\op M^S:\;\frac{\partial\;}{\partial u}f_1(0,x_2,x_3;\phi_1(u))=0}\;.
\end{equation}
%
Returning to \eref{blowup} and setting $\eps=0$ gives the slow critical subsystem (sometimes the called {\it reduced problem})
\begin{equation}\label{blowup0}
(0,\dot x_2,\dot x_3)=\bb{f_1(0,x_2,x_3;\phi_1(u)),\;f_2(0,x_2,x_3;\phi_1(u)),\;f_3(0,x_2,x_3;\phi_1(u))}\;.
\end{equation}
which defines dynamics inside the critical manifold $\op M^S/\op L$, on the original timescale. 

In our present context we will call $\op M^S$ the {\it sliding manifold} $\op M^S$, and refer to the dynamics in \eref{blowup0} on $\op M^S$ as {\it sliding dynamics}, and to its solutions as {\it sliding modes}; this is due to their conjugacy to sliding modes in the piecewise smooth system \eref{fns}, as proven in \cite{t07,j15douglas}. In \cite{t07} it is shown that the slow dynamics of \eref{blowup} on $\op M^S/\op L$, is conjugate to the Filippov sliding dynamics of the piecewise smooth system \eref{ffil}, and in \cite{j15douglas} this result is extended to include the generalization \eref{fnon}.




We now apply these ideas to the normal form of the two-fold singularity. The final step in our analysis will be to transform \eref{blowup} into the normal form of a folded singularity, but we shall find that this is only possible if we regularize \eref{fnon} with ${\bf g}\neq0$.

\section{The unperturbed system}\label{sec:2fdummy}

In this section we show that the Sotomayor-Teixeira regularization of the two-fold singularity has a sliding manifold $\op M^S$ as defined in \sref{sec:blowup}, but that the non-hyperbolic line $\op L$ lies along the fast direction (the $u$-axis), and hence in a degenerate position with respect to the flow. 
We show that this degeneracy is broken by perturbations of the form \eref{fnon} in \sref{sec:2fpert}. 

We first express the two-fold normal form as a convex combination by substituting \eref{2f} into \eref{ffil}, giving
\begin{eqnarray}\label{2fg0}
(\dot x_1,\dot x_2,\dot x_3)&=&\frac{1+\lambda}2(-x_2,a_1,b_1)+\frac{1-\lambda}2(x_3,b_2,a_2)\\
&:=&\bb{f_1(x_1,x_2,x_3;\lambda),\;f_2(x_1,x_2,x_3;\lambda),\;f_3(x_1,x_2,x_3;\lambda)}\;.\nonumber
\end{eqnarray}
We then regularize this by replacing $\lambda\mapsto\phi_\eps(x_1)$ for small $\eps$, constituting a Sotomayor-Teixeira regularization of \eref{fns}, and let $u=x_1/\eps$ to obtain the two timescale system
\begin{eqnarray}\label{2fg0phi}
(\eps\dot u,\dot x_2,\dot x_3)&=&\frac{1+\phi_1(u)}2(-x_2,a_1,b_1)+\frac{1-\phi_1(u)}2(x_3,b_2,a_2)\\
&=&\bb{f_1(\eps u,x_2,x_3;\phi_1(u)),\;f_2(\eps u,x_2,x_3;\phi_1(u)),\;f_3(\eps u,x_2,x_3;\phi_1(u))}\;.\nonumber
\end{eqnarray}
By \eref{Mdef}, the sliding manifold $\op M^S$ 
is given by
\begin{equation}\label{M2f}
\op M^S=\cc{\;(u,x_2,x_3)\in\mathbb R^3:\;\phi_1(u)=\frac{x_3-x_2}{x_3+x_2},\;|u|<1\;}\;.
\end{equation}
The condition $|\phi_1(u)|<1$ by \eref{Peps} implies that $\op M^S$ exists only for $0<x_2x_3$ or for $x_2=x_3=0$. 
%
%
By expressing $\op M^S$ implicitly as the zero contour of the smooth function $(x_3+x_2)\phi_1(u)+(x_2-x_3)$, we see that it is a smooth surface which twists over near $x_2=x_3=0$. It consists of two normally hyperbolic branches, one attractive in $x_2,x_3>0$ since $\partial f_1/\partial u=-(x_3+x_2)\phi_1'(u)/2<0$ (using the fact that $\phi_1'(u)>0$ by \eref{Peps}), and one repulsive in $x_2,x_3<0$ since $\partial f_1/\partial u=-(x_3+x_2)\phi_1'(u)/2>0$. 
The two branches are connected at $x_2=x_3=0$ along the non-hyperbolic set, found from \eref{Ldef} to be
\begin{equation}\label{L}
\op L=\cc{(u,x_2,x_3)\in\op M^S:\;x_2=x_3=0\;}\;.
\end{equation}
This line segment $\op L$, at which the attracting and repelling branches of $\op M^S$ intersect, constitutes the regularization of the two-fold singularity ($x_1=x_2=x_3=0$ in \eref{fns}), now existing for all $|u|<1$ at $x_2=x_3=0$ on $\op M^S$. 
\Fref{fig:2fu} shows an example of the piecewise smooth system (i), its regularization showing $\op M^S$ and $\op L$ in (ii), which is then rotated in to show $\op L$ more clearly (iii). 
\begin{figure}[h!]\centering\includegraphics[width=\textwidth]{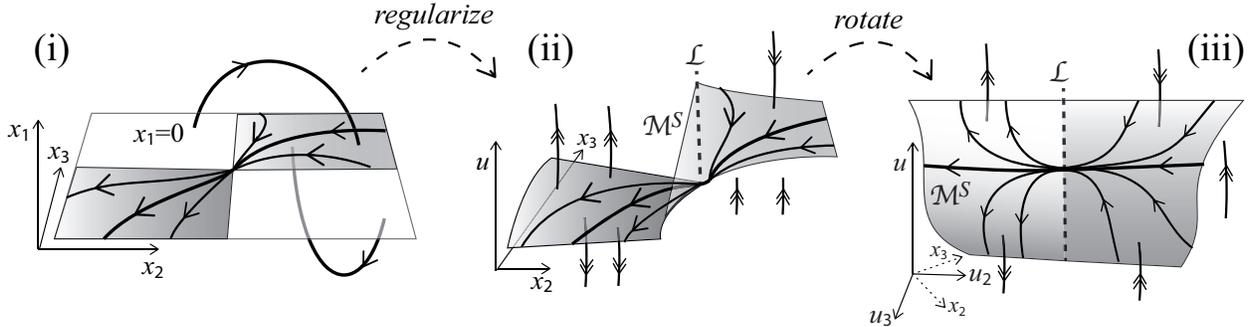}
\vspace{-0.3cm}\caption{\small\sf Regularizing the unperturbed system \eref{2fg0}, for the example of an invisible two-fold. (i) The flow directions outside $x_1=0$  create an attracting sliding region in $x_2,x_3>0$ and repelling sliding region in $x_2,x_3<0$. (ii) The regularization of $x_1=0$, replacing $x_1$ with a fast variable $u=x_1/\eps$, where the sliding regions create a critical manifold $\op M^S$ (shaded), hyperbolic except along the vertical line $\op L$, which aligns with the fast (double arrowed) $u$ dynamics. (iii) The dynamics in the manifold is best viewed along the $u_3$ axis of rotated coordinates $u_2=x_2+x_3$, $u_3=x_2-x_3$. }\label{fig:2fu}\end{figure}

We then have the main result of this section. 
 
\begin{proposition}
In the regularization of the normal form two-fold singularity \eref{2fg0phi}, the non-hyperbolic set $\op L$ of the sliding manifold $\op M^S$ lies everywhere tangent to the coordinate axis of the fast variable. 
\end{proposition}

\begin{proof}
The non-hyperbolic set $\op L$ forms a line with tangent vector $e_{\op L}=(1,0,0)$ in the space of $(u,x_2,x_3)$, which means it lies everywhere parallel to the fast $u$-coordinate axis of the two timescale system \eref{2fg0phi}. This is related to the fact that all derivatives of $f_1$ with respect to the fast variable $u$ vanish along $\op L$, not only the first derivative $\partial f_1/\partial u=-(x_3+x_2)\phi'(u)=0$ which defines $\op L$ as the set $x_2=x_3=0$ (since $\phi'(u)\neq0$ for $|u|<1$ by \eref{Peps}), but also all higher derivatives $\partial^rf_1/\partial u^r=-(x_3+x_2)\phi^{r}(u)=0$ for any $r>1$. Thus this constitutes an infinite codimension degeneracy. 

\end{proof}

In the literature on smooth two timescale systems, the connection of attracting and repelling branches of a slow invariant manifolds has been well studied, leading to a generic canonical form and requisite non-degeneracy conditions as described in \cite{w05}. In the present notation, the non-degeneracy of $\op M^S$ along $\op L$ requires the conditions
\begin{equation}\label{wechscan}
\begin{array}{ccccccc}
f_1=0\;,&\quad&\frac{\partial f_1}{\partial u}=0\;,&\quad&
\frac{\partial f_1\;\;\;\;}{\partial x_{2,3}}\neq0\;,&\quad&\frac{\partial^2 f_1}{\partial u^2}\neq0\;,\end{array}
\end{equation}
the first three of which are satisfied on $\op L$ as given by \eref{L}, while the fourth is violated everywhere on $\op L$. Therefore the degeneracy of $\op L$ prevents us relating \eref{2fg0phi} to the canonical form of such singularities in slow-fast systems.

In light of this problem, the papers \cite{t07,tls08,hk15} take a different approach to handling the system \eref{2fg0phi}. While the degeneracy above is not remarked on explicitly, the difficulties that arise from it are, and are tackled by re-scaling the local variables to prove that canards persist for perturbations within the Sotomayor-Teixeira regularization. Here we will instead permit perturbations that constitute a more general regularization, and in doing so we are able to obtain the canonical form satisfying \eref{wechscan}.


\section{The perturbed system}\label{sec:2fpert}

We will now show that a certain perturbation breaks the degeneracy of the system in the previous section. 

To achieve this, first observe that adding constant terms or functions of the coordinates $(x_1,x_2,x_3)$ to \eref{2fg0phi} would only move the set $\op L$ in the $(x_2,x_3)$ plane, not remove its degeneracy, easily seen since the derivatives $$\frac{\partial^r}{\partial u^r} f_1(0,x_2,x_3,\phi_1(u))=\mbox{$\frac{f_1(0,x_2,x_3,+1)-f_1(0,x_2,x_3,-1)}2$}\phi^{r}(u)\qquad\mbox{for}\quad r>0$$ would still vanish on $\op L$, where $f_1(0,x_2,x_3,+1)-f_1(0,x_2,x_3,-1)=0$. 
The only recourse to break the degeneracy, specifically to give $\frac{\partial^2 f_1}{\partial u^2}\neq0$, is therefore to add terms nonlinear in $\phi_1$ to \eref{2fg0phi}, and hence terms nonlinear in $\lambda$ to \eref{2fg0}. Anything we add to the function $f_1$ in \eref{2fg0} must still give \eref{2f}, so it must vanish outside the switching surface $x_1=0$, i.e. be a perturbation in the form \eref{fnon}. We shall show that perturbing $\dot x_1$ with a small term proportional to $\lambda^2-1$ is sufficient for structural stability. Perturbing $\dot x_2$ or $\dot x_3$ is neither necessary nor sufficient, therefore we leave them unaltered. 
%

The perturbed system we consider, applying \eref{fnon} to the normal form \eref{2f} with ${\bf g}=(\alpha,0,0)$, is
\begin{eqnarray}\label{2Fp}
\bb{\dot x_1,\dot x_2,\dot x_3}&=&\frac{1+\lambda}2\bb{-x_2,a_1,b_1}+\frac{1-\lambda}2\bb{x_3,b_2,a_2}+(1-\lambda^2)(\alpha,0,0)\\
&:=&\bb{f_1(x_1,x_2,x_3;\lambda),f_2(x_1,x_2,x_3;\lambda),f_3(x_1,x_2,x_3;\lambda)}\;,\nonumber
\end{eqnarray}
where $\alpha$ is a constant. 
The regularization, replacing $\lambda\mapsto\phi_\eps(x_1)$ for small $\eps$ and letting $u=x_1/\eps$, becomes
\begin{eqnarray}\label{2foldblow}
(\eps\dot u,\dot x_2,\dot x_3)&=&\frac{1+\phi_1(u)}2\bb{-x_2,a_1,b_1}+\frac{1-\phi_1(u)}2\bb{x_3,b_2,a_2}+\alpha(1-\phi_1^2(u))\bb{1,0,0}\\
&:=&\bb{f_1(x_1,x_2,x_3;\lambda),\;f_2(x_1,x_2,x_3;\lambda),\;f_3(x_1,x_2,x_3;\lambda)}\;,\nonumber
\end{eqnarray}
which is an $\alpha$-perturbation of \eref{2fg0phi}. Our main result is then:

\begin{proposition}\label{thm:non}
The regularization \eref{2foldblow} of the normal form two-fold singularity \eref{2f}, using \eref{fnon} with ${\bf g}=(\alpha,0,0)$, can be transformed into the canonical form for the folded singularity \cite{w05}, namely
\begin{eqnarray*}
\eps\dot{ x_1}&=& x_2+ x_1^2+\ord{\eps x_1,\eps x_3, x_1 x_3}\\
\dot{ x}_2&=& p x_3+ q x_1+\ord{ x_3^2, x_1 x_3}\\
\dot{ x}_3&=& r+\ord{ x_3, x_1}
\end{eqnarray*}
provided $\alpha\neq0$ for small $\eps>0$, where $p,q,r,$ are real constants, and provided 
the conditions $\hf(b_1-b_2)\le1=a_1=-a_2$ or $\hf(b_1-b_2)\ge-1=a_1=-a_2$ do not hold. 
\end{proposition}
\noindent It turns out that the case excluded by the conditions $\pm\hf(b_1-b_2)\le1=\pm a_1=\mp a_2$ is that in which there are no canards or faux-canards, i.e. no orbits of the sliding flow passing through the singularity. 

We shall prove the proposition by way of three lemmas, establishing first the non-degeneracy of $\op L$, second locating a singularity along $\op L$, and finally using these to derive new local coordinates in which $\op M^S$ becomes a simple parabolic surface.
 
The sliding manifold, found by applying \eref{Mdef} to \eref{2foldblow}, is now the set
\begin{equation}\label{2foldMSp}
\op M^S=\cc{(u,x_2,x_3)\in\mathbb R^3\;:\;|u|<1,\;\frac{1-\phi_1(u)}2x_3-\frac{1+\phi_1(u)}2x_2+\alpha(1-\phi_1^2(u))=0}\;,
\end{equation}
which is normally hyperbolic except on the set given by applying \eref{Ldef} to \eref{2foldblow}, 
\begin{equation}\label{Lpert}
\op L=\cc{(u,x_2,x_3):\;\phi_1(u)=2\frac{2\alpha+x_3-x_2}{x_3+x_2}=-\frac{x_3+x_2}{4\alpha}}\;.
\end{equation}
Solving the conditions in \eref{Lpert} we can express $\op L$ in paramteric form as
\begin{equation}\label{Lparam}
(u,x_2,x_3)=\op L(y_1):=\bb{u,\alpha(\phi_1(u)-1)^2,\; -\alpha(\phi_1(u)+1)^2}\;.
\end{equation}
This gives us the first result as follows. 

\begin{lemma}\label{thm:l1}
The non-hyperbolic set $\op L$ is transverse to the fast direction of \eref{2foldblow}. 
\end{lemma}

\begin{proof}
By differentiating \eref{Lparam} with respect to $u$, we find that the curve $\op L$ has tangent vector
$e_{\op L}
=\bb{1\;,\;2\alpha(\phi_1(u)-1)\phi_1'(u)\;,\;-2\alpha(\phi_1(u)+1)\phi_1'(u)}$, which for all $|\phi_1(u)|\le1$ is transverse to the coordinate axes provided $\alpha\neq0$. 
\end{proof}


While the non-hyperbolic curve $\op L$ is now in a general position with respect to the fast variable, generically there may exist a new singularity along $\op L$, where the flow's projection along the $\lambda$-direction onto the nullcline $f_1=0$ is indeterminate. This is the so-called {\it folded singularity}, defined in the following lemma. 

\begin{lemma}\label{thm:l2}
For the values of the constants $a_1,a_2,b_1,b_2,$ given in Proposition \ref{thm:non}, there exists an isolated singularity of the flow along the non-hyperbolic set $\op L$, where the projection of the slow flow onto $\op M^S$ lies tangent to $\op L$. 
\end{lemma}

\begin{proof}
Let us consider the slow critical subsystem, obtained by letting $\eps=0$ in \eref{2foldblow}, 
\begin{eqnarray}\label{2fslow}
(0,\dot x_2,\dot x_3)&=&\frac{1+\phi_1(u)}2\bb{-x_2,a_1,b_1}+\frac{1-\phi_1(u)}2\bb{x_3,b_2,a_2}+\alpha(1-\phi_1^2(u_s))\bb{1,0,0}\\\nonumber
&=&\bb{f_1(0,x_2,x_3;\phi_1(u)),f_2(0,x_2,x_3;\phi_1(u)),f_3(0,x_2,x_3;\phi_1(u))}\;.
\end{eqnarray}
Since $\op M^S$ is the surface where $f_1=0$, a solution of \eref{2fslow} that remains on $\op M^S$ for an interval of time satisfies $\dot f_1=0$. We can find $\dot u$ on $\op M^S$ using the chain rule, writing 
$\dot f_1=\bb{\dot u,\dot x_2,\dot x_3}\cdot\bb{\partial/\partial u,\partial/\partial x_2,\partial/\partial x_3}f_1=\dot u\frac{\partial f_1}{\partial u}+(f_2,f_3)\cdot\frac{\partial f_1\quad}{\partial(x_2,x_3)}=0$, which rearranges to $\dot u=-(f_2,f_3)\cdot \frac{\partial f_1\quad}{\partial(x_2,x_3)}/\bb{\frac{\partial f_1}{\partial u}}$
. Thus $\dot u$ is indeterminate on $\op M^S$ at points where the numerator and denominator of this vanish, or in full, where
\begin{equation}\label{fsing}
0=f_1=\frac{\partial f_1}{\partial u}=(f_2,f_3)\cdot \frac{\partial f_1\quad}{\partial(x_2,x_3)}\;.
\end{equation}
These three conditions define an isolated singularity on $\op L\subset\op M^S$. 
Denoting the value of $f_i$ at the singularity as $f_{is}$, and solving \eref{fsing}, which constitutes finding a point on $\op L$ given by \eref{Lpert} such that
\begin{eqnarray}\label{f23s}
0&=&(f_{2s},f_{3s})\cdot\bb{-\frac{1+\phi_1(u_s)}2,\frac{1-\phi_1(u_s)}2}\nonumber\\
&=&\mbox{$\bb{\frac{a_1+b_2}2+\frac{a_1-b_2}2\phi_1(u_s)\;,\;\frac{b_1+a_2}2+\frac{b_1-a_2}2\phi_1(u_s)}\cdot\bb{-\frac{1+\phi_1(u_s)}2,\frac{1-\phi_1(u_s)}2}$}\;,
\end{eqnarray}
we find that the folded singularity lies at $(u,x_2,x_3)=(u_s,x_{2s},x_{3s})$ where
\begin{equation}\label{fsingu}
\phi_1(u_s)=
\frac{-\frac{a_1+a_2}{b_1-b_2}\pm\sqrt{1+\frac{4a_1a_2}{(b_1-b_2)^2}}}{1+\frac{a_1-a_2}{b_1-b_2}}\;,
\quad x_{2s}=\alpha(\phi_1(u_s)-1)^2\;,\quad x_{3s}=-\alpha(\phi_1(u_s)+1)^2\;.
\end{equation}
Noting that $a_1$ and $a_2$ in the normal form just take values $\pm1$, and recalling that $\phi_1(u_s)$ is monotonic on $|u|<1$ by \eref{Peps}, we have:
\begin{itemize}
\item in the case $a_1=a_2=1$, we have $\phi_1(u_s)=\frac{-2}{b_1-b_2}\pm\sqrt{1+\frac{4}{(b_1-b_2)^2}}$, implying that there exists a unique solution $u_s\in(-1,+1)$ 
for any $b_1$ and $b_2$ (the positive root for $b_1>b_2$, the negative root for $b_1<b_2$);
\item in the case $a_1=a_2=-1$, we have $\phi_1(u_s)=\frac{2}{b_1-b_2}\pm\sqrt{1+\frac{4}{(b_1-b_2)^2}}$, implying that there exists a uniquesolution $u_s\in(-1,+1)$
 for any $b_1$ and $b_2$ (the positive root for $b_1<b_2$, the negative root for $b_1>b_2$);
\item in the case $a_1=-a_2=1$, we have $\phi_1(u_s)=\pm\sqrt{\frac{b_1-b_2-2}{b_1-b_2+2}}$, implying that there exist two solutions $u_s\in(-1,+1)$ 
for $b_1-b_2>2$, and no points otherwise. 
\item in the case $a_1=-a_2=-1$, we have $\phi_1(u_s)=\pm\sqrt{\frac{b_1-b_2+2}{b_1-b_2-2}}$, implying that there exist two solution $u_s\in(-1,+1)$ 
for $b_1-b_2<-2$, and no points otherwise. 
\end{itemize}

\end{proof}

This lemma establishes the existence of at least one unique folded singularity on $\op L$ in the cases listed in Proposition \ref{thm:non}. 
In the cases where $u_s$ is unique we proceed directly to the steps that follow below. In the cases where $u_s$ can take two values we can proceed with the following analysis about each value, and will obtain different constants in the final local expression, i.e. a different folded singularity corresponding to each $u_s$. In the cases when $u_s$ does not exist, no equivalence can be formed; these are the cases when the the two-fold's sliding portrait is of focal type (see \cite{jc12}), and there exists no canards since orbits wind around the two-fold but never enter or leave it. So excluding those cases $a_1=-a_2=1$ with $b_1-b_2\le2$ and $a_1=-a_2=-1$ with $b_1-b_2\ge-2$, we proceed with the final step in proving proposition \ref{thm:non}. 

\begin{lemma}\label{thm:l3}
Coordinates can be defined in which the folded singularity of \eref{2foldblow} lies at the origin, and $\op L$ lies along a coordinate axis corresponding to a slow variable. 
\end{lemma}

\begin{proof}
Taking a valid solution of $u_s$ from \eref{fsingu} for $|u|<1$, a translation puts the singularity at the origin of the new coordinates
\begin{equation}
y_1=\phi_1(u)-\phi_1(u_s)\;,\qquad y_2=x_2-x_{2s}\;,\qquad y_3=x_3-x_{3s}\;.
\end{equation}
Then $f_1$ becomes
\begin{equation}\label{fz}
f_1=-\frac{1+\phi_1(u_s)}2y_2   +\frac{1-\phi_1(u_s)}2y_3    -\bb{\frac{y_3+y_2}2+\alpha y_1}y_1\;,
\end{equation}
found by using \eref{f23s}-\eref{fsingu} to ensure that terms involving $x_{2s}$ and $x_{3s}$ vanish. 
To find coordinates in which $\op L$ lies along a coordinate axis, from \eref{Lparam} we can obtain the $y_1$-parameterized expression for $\op L$, 
\begin{equation*}
(y_1,y_2,y_3)=\op L(y_1):=\bb{y_1,-\alpha y_1(2 -2\phi_1(u_s) - y_1),\; -\alpha y_1(2+2\phi_1(u_s) + y_1)}\;,
\end{equation*}
and re-arrange this to take $y_3$ as a parameter, expressing $\op L$ as $(y_1,y_2)=\bb{y_{1L}(y_3),y_{2L}(y_3)}$, where
\begin{equation}
\bb{\begin{array}{c}y_{1L}(y_3)\\y_{2L}(y_3)
\end{array}}:=\bb{ \begin{array}{c}  -1-\phi_1(u_s)+\sqrt{(1+\phi_1(u_s))^2-y_3/\alpha} \\ -y_3-4\alpha (-1-\phi_1(u_s)+\sqrt{(1+\phi_1(u_s))^2-y_3/\alpha})
\end{array} }\;.
\end{equation}
The derivatives of these functions are needed to evaluate the vector field components below, these are
\begin{equation}
y_{1L}'(y_3)
=\frac{-1/2\alpha}{1+\phi_1(u_s)+y_{1L}(y_3)}\;,\quad 
y_{2L}'(y_3)
=\frac{1-\phi_1(u_s)-y_{1L}(y_3)}{1+\phi_1(u_s)+y_{1L}(y_3)}\;.
\end{equation}
We can then rectify $\op L$ to lie along some $z_3$ axis by defining new coordinates
\begin{equation}
z_1=y_1-y_{1L}(y_3)\;,\qquad z_2=y_2-y_{2L}(y_3)\;,\qquad z_3=y_3\;.
\end{equation}
The original vector field components can then be written as
\begin{equation}\label{ove}
\begin{array}{rcl}
f_1&=&-\frac{1+\phi_1(u_s)+y_{1L}}2z_2  - \alpha z_1^2-z_2z_1/2\;,\\
f_2&=&f_{2s}+\bb{z_1+y_{1L}(z_3)}\frac{\partial f_{2s}}{\partial\phi_1}=f_{2s}+\ord{z_1,z_3}\;,\\
f_3&=&f_{3s}+\bb{z_1+y_{1L}(z_3)}\frac{\partial f_{3s}}{\partial\phi_1}=f_{3s}+\ord{z_1,z_3}\;.
\end{array}
\end{equation}
With a little algebra we find that
\begin{eqnarray*}
\eps\dot z_1&=&\eps\dot y_1-\eps\dot y_3y_{1L}'(y_3)\\
&=&\frac{1+\phi_1(u_s)}2\bb{\frac{\eps f_{3s}}{\alpha(1+\phi_1(u_s))^2}-z_2\phi_1'(u_s)}-\alpha z_1^2\phi_1'(u_s)+\ord{\eps z_3,\eps z_1,z_2z_3,z_2z_1,z_1^3}
\end{eqnarray*}
%
A small shift $\tilde z_2=z_2-\frac{\eps f_{3s}}{\alpha(1+\phi_1(u_s))^2\phi_1'(u_s)}$ yields, after some lengthy but straightforward algebra, using the relations in \eref{f23s} and \eref{ove} to show that any terms not propotional to $z_1$ or $z_3$ vanish, 
\begin{eqnarray*}
\dot{\tilde z}_2&=&f_2-\dot z_3y_{2L}'(z_3)\\
&=&qz_1+\frac p\alpha z_3+\ord{z_3^2,z_1z_3}
\end{eqnarray*}
where
$$q=\frac{\partial f_{2s}}{\partial\phi_1}-\frac{\partial f_{3s}}{\partial\phi_1}\frac{1-\phi_1(u_s)}{1+\phi_1(u_s)}\;,\qquad
p=-\frac{2f_{3s}+q(1+\phi_1(u_s))^2}{2(1+\phi_1(u_s))^2}\;.$$
The last thing to do is just scaling. Collecting everything together so far we have
\begin{equation*}
\begin{array}{rcl}
\eps\dot z_1&=&(d_1{\tilde z}_2-\alpha z_1^2)\phi_1'(u_s)+\ord{\eps z,\eps z_3,z_1z_3}\\
\dot {\tilde z}_2&=&\frac p\alpha z_3+qz_1+\ord{z_3^2,z_1z_3}\\
\dot z_3&=& f_{3s}+\ord{z_3,z_1}
\end{array}
\end{equation*}
where $d_1=-\hf(1+\phi_1(u_s))$. 
Defining new variables $\tilde x_1=\sqrt{|\alpha|\phi_1'(u_s)}z_1$, $\tilde x_2=-{\rm sign}(\alpha)d_1\phi_1'(u_s) {\tilde z}_2$, $\tilde x_3=-{\rm sign}(\alpha)z_3$, and $\tilde t=-{\rm sign}(\alpha)t$, gives
\begin{equation}\label{wechs}
\begin{array}{rcl}
\eps\dot{\tilde x}_1&=&\tilde x_2+\tilde x_1^2+\ord{\eps x_1,\eps\tilde x_3,\tilde x_1\tilde x_3}\\
\dot{\tilde x}_2&=&\tilde p\tilde x_3+\tilde q\tilde x_1+\ord{\tilde x_3^2,\tilde x_1\tilde x_3}\\
\dot{\tilde x}_3&=&\tilde r+\ord{\tilde x_3,\tilde x_1}
\end{array}
\end{equation}
where
\begin{equation}\label{wp}
\begin{array}{l}
\tilde r= f_{3s}
\;,\quad\tilde p
=-\frac{1}{4|\alpha|\phi_1'(u_s)}\bb{f_{2s}+ f_{3s}-2\tilde q\sqrt{|\alpha|\phi_1'(u_s)}}\;,\\
\tilde q
=\frac{-1}{2\sqrt{|\alpha|\phi_1'(u_s)}}\bb{(\phi_1(u_s)+1)\frac{\partial f_{2s}}{\partial\phi_1}+(\phi_1(u_s)-1)\frac{\partial f_{3s}}{\partial\phi_1}}
\;.
\end{array}
\end{equation}
Omitting the tildes, this is the result in the lemma and in Proposition \ref{thm:l3}, clearly valid only for $\alpha\neq0$. 

\end{proof}

\noindent{\bf Remarks on the singularity}

A glance at the papers \cite{w05,w12,des10} reveals what a charismatic singularity lies hidden in the dynamics of the two-fold, waiting to be released when the piecewise smooth system is perturbed by simulations that smooth, regularize, or otherwise approximate the discontinuity. As for the two-fold itself in \sref{sec:2fold}, a detailed description is beyond our interest here and can be pursued in future work, but as a guide we shall briefly gather together the main points from the literature. 


For sufficiently small $\eps>0$, by standard results of geometric singular perturbation theory \cite{f79}, there exist invariant manifolds $\op M^S_\eps$ in the neighbourhood of $\op M^S/\op L$, on which the dynamics is topologically equivalent to the sliding dynamics found above. Trajectories that pass close to the singularity, or more precisely, close to the folded singularity on the set $\op L$, may persist in following the  manifold $\op M^S$ from its stable to unstable branches, while other nearby trajectories will veer wildly away, their fate sensitive to initial conditions and proximity to primary canard orbits (those which persist along both branches of $\op M^S$ throughout the local region).

\Fref{fig:2fs} shows an example of the perturbed system and its regularization for each flavour of two-fold in (i) (corresponding to those in \fref{fig:2folddet}), followed by their regularization (ii), and a rotation (iii) to show the phase portrait around the set $\op L$ more clearly (similar to \fref{fig:2fu}). 
In the most extreme case, the folded node, the original phase portrait contains infinitely many intersecting trajectories traversing the singularity, while the perturbed system splits these into distinguishable orbits, a finite number of which asymptote to the attracting and repelling branches of the critical manifold.  

\begin{figure}[h!]\centering\includegraphics[width=\textwidth]{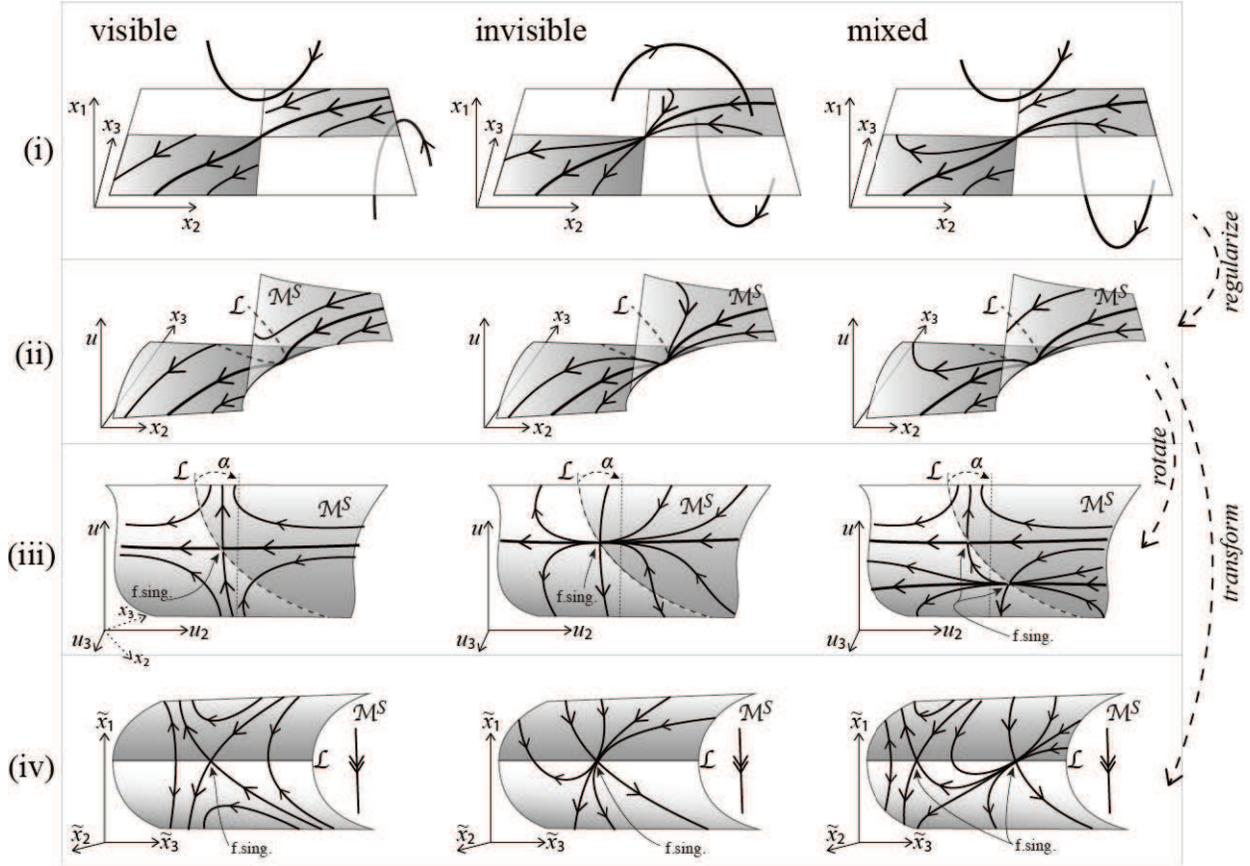}
\vspace{-0.3cm}\caption{\small\sf blowing up the perturbed ($\alpha\neq0$) system, for examples of each flavour of two-fold. Labelling as in \fref{fig:2fu}. Note in the regularization (ii) that $\op L$ is now a curve. Rotating around the $u$ axis in (iii) we can see the attracting branch (upper right segment) and repelling branch (lower left segment) of the sliding manifold $\op M^S$ (shaded), connected by $\op L$. The folded singularity (f.sing.) appears along $\op L$, two in the case of mixed visibility, recognised as having a phase portrait that resembles a saddle or node if we reverse time in the repelling branch of $\op M^S$. In (iv) we sketch the corresponding phase portraits in the slow-fast system \eref{wechs}. 
}\label{fig:2fr}\end{figure}

Like the different kinds of two-fold, there are different classes of folded singularity, and their classification depends on the slow dynamics inside $\op M^S$. From the expressions \eref{wechs}-\eref{wp} we see that the class therefore depends not only on the constants $a_1,a_2,b_1,b_2,$ of the original piecewise smooth system, but also on the `hidden' parameter $\alpha$.  
%

The classification scheme is fairly simple, and can be used to verify the dynamics on $\op M^S$ seen in \fref{fig:2fr}. The projection of the system \eref{wechs} onto $\op M^S$, found by differentiating the condition $0=\tilde x_2+\tilde x_1^2$ with respect to time to give $0=\tilde b\tilde x_3+\tilde c\tilde x_1+2\tilde x_1\dot{\tilde x}_1+\ord{\tilde x_3^2,\tilde x_1\tilde x_3}$, is
$$\bb{\begin{array}{c}\dot{\tilde x}_1\\\dot{\tilde x}_3\end{array}}=\frac1{-2\tilde x_1}\cc{\bb{\begin{array}{cc}\tilde c&\tilde b\\-2\tilde a&0\end{array}}\bb{\begin{array}{c}\tilde x_1\\\tilde x_3\end{array}}+\ord{\tilde x_3^2,\tilde x_1\tilde x_3}}\;.$$
A classification then follows by neglecting the singular prefactor $1/2\tilde x_1$ and considering whether the phase portrait is that of a focus, a node, or a saddle. This is determined by the $2\times2$ matrix Jacobian, which has 
trace $\tilde c$, determinant $2\tilde a\tilde b$, and eigenvalues $\hf(\tilde c\pm\sqrt{\tilde c^2-8\tilde a\tilde b})$. 
This will not be the true system's phase portrait because the time-scaling from the $1/2\tilde x_1$ factor is positive in the attractive branch of $\op M^S$, negative (time-reversing) in the repulsive branch, and divergent at the singularity (turning infinite time convergence to the singularity into finite time passage {\it through} the singularity). The effect of this is to `fold' together attracting and repelling pairs of each equilibrium type, so each equilibrium becomes a `folded-equilibrium', forming a continuous bridge between branches of $\op M^S$.

As a result the flow on $\op M^S$ is a folded-saddle if $\tilde a\tilde b<0$, a folded-node if $0<8\tilde a\tilde b<\tilde c^2$, and a folded-focus if $\tilde c^2<8\tilde a\tilde b$. Canard cases occur for $\tilde c>0$ and faux canard for $\tilde c<0$. In \fref{fig:2fr} we show the result of regularizing the piecewise smooth system for an example of each type of two-fold that exhibits determinacy-breaking (those from \fref{fig:2folddet}). In the visible two-fold the singularity becomes a folded-saddle, in the invisible case it becomes a folded-node, while the mixed case becomes a pair consisting of one folded-saddle and one folded-node.

One may ask why certain cases were excluded by the proposition above. The excluded cases were those in which no canards exist in the slow-fast system. Canards occur when transversal intersections exist between the attracting and repelling branches of the slow manifolds. If no such intersections exist, the critical system possesses no folded singularities and hence is excluded from Proposition \ref{thm:non}. 
Hence the omission of these cases is consistent, and a posteori it is obviously necessary, in the equivalence sought in the proposition. 

%

With our main result proven, we conclude with two sections relevent to the study above, which help elucidate certain ideas that have arisen lately in the study of piecewise smooth dynamical systems, which are particularly relevant to the study of singularities like the two-fold.

\section{``Discontinuity blow-up'', an approach to nonsmooth systems without regularization}

Many of the relations in the analysis above had to written implicitly in terms of the regularization function $\phi_1(u)$, having introduced $\lambda=\phi_\eps(x_1)=\phi_1(u)$. We could have proceded instead by studying the dynamics of $\lambda$ directly, omitting reference to a smoothing function $\phi_\eps$ altogether. This alternative approach to nonsmooth systems was discussed in \cite{j13error}, and makes the results above somewhat more concise and explicit, but a couple of propositions are required to show that standard concepts from geometric singular perturbation theory can be applied on $(\lambda,x_2,x_3)$ in place of $(u,x_2,x_3)$. Having followed the conventional route above, we provide the basic results needed for this alternative route here. 

Firstly we require a dynamical system on $\lambda$. 
\begin{proposition}\label{thm:ldot}
The dynamics of $\lambda$ is given by
\begin{equation}\label{le}
\e(\lambda,\eps)\dot\lambda=\dot x_1={\bf f}({\bf x},\lambda)\cdot\nabla x_1\;
\end{equation}
such that $\e(\lambda,\eps)\ll1$, where $\e$ denotes a continuous positive function and $\eps$ a small parameter, with $0<\eps<\eps^*\ll1$ and $\lambda\in(\phi_{\eps*}^-,\phi_{\eps*}^+)$, in terms of constants $\eps^*$ and $\phi_{\eps*}^\pm$ that satisfy $\phi_{\eps*}^\pm\rightarrow\pm1$ as $\eps^*\rightarrow0$. 
\end{proposition}

\begin{proof}
Consider regularizing the vector field \eref{fnon} by replacing $\lambda$ with a differentiable sigmoid function $\phi_\eps(x_1)$ as defined in \eref{Peps}. We shall use the relation $\lambda=\phi_\eps(x_1)$ to derive a dynamical system on $\lambda$. 
Differentiating $\lambda=\phi_1(u)$ with respect to $t$ gives 
\begin{eqnarray}\label{ldot}
\dot\lambda
&=&\frac{\dot x_1}\eps\phi_\eps'(x_1)\qquad\mbox{for}\quad|x_1|<\eps\;.
\end{eqnarray}
Considering a variable $u=x_1/\eps$ we see that, according to \eref{Peps}, the function $\phi_\eps(x_1)=\phi_1(u)$ and its derivative $\eps\phi_\eps'(x_1)=\phi_1'(u)$ are smooth with respect to $u$ in the limit $\eps\rightarrow0$. Moreover $\phi_1'(u)$ is strictly positive because $\phi_1(u)$ is strictly increasing, and $\phi_1'(v/\eps)$ only becomes small (or vanishing) for $|x_1|/\eps>1$. So the quantity $\eps/\phi_1'$ is small and nonzero for $|x_1|/\eps\le1$, and using it we define a fast timescale $\tau=t\phi_\eps'(x_1)/\eps$. Since $\phi_1(u)$ is differentiable and monotonic for $|u|<1$, it has an inverse $\psi(\lambda)$ such that $\psi(\phi_1(u))=u$, and we can define a function 
\begin{equation}
\e(\lambda,\eps):=\eps/\phi_1'(\psi(\lambda))\;,\qquad\mbox{for}\quad|\lambda|<1\;.
\end{equation}
That this quantity is small is shown as follows: the function $\phi_1(u)$ varies differentiably over interval on which its extremal values are $\phi_1(\pm1)=\pm1$, therefore there exists a point $u_*$ where $\phi'(u_*)=\frac{\phi_1(+1)-\phi_1(-1)}{(+1)-(-1)}=1$, and by continuity since $\phi_1'(\pm1)=0$, there exist two points $u_\eps^\pm$ where $\phi'(u_\eps^\pm)=\pm\eps$ for $0<\eps<1$, and moreover an interval $u_\eps^-<u<u_\eps^+$ such that $\phi_1'(u)>\eps$. Fix some $\eps_*$ such that $0<\eps_*\ll1$, then $\eps/\phi_1(u)<1$ for $u_{\eps*}^-<u<u_{\eps*}^+$, and 
$$\lim_{\eps\rightarrow0}\e(\lambda,\eps)=0\;,$$
so that $\e(\lambda,\eps)\ll1$ for $\eps\ll\eps_*$ and $u\in\bb{u_{\eps*}^-,u_{\eps*}^+}$. 

By \eref{ldot} we therefore have the dynamical equation 
$\e(\lambda,\eps)\dot\lambda=\dot x_1={\bf f}({\bf x},\lambda)\cdot\nabla x_1$
in the proposition. 
\end{proof}

This proposition identifies $\lambda$ as a fast variable inside $\lambda\in(-1,+1)$ (more strictly for $\lambda\in(\phi_{\eps*}^-,\phi_{\eps*}^+)$ where $\phi_{\eps*}^\pm=\phi_1(u_{\eps*}^\pm)$, and $\eps^*$ is arbitrarily small but nonzero). 
When $\lambda$ is set-valued on $v=0$ with $\eps=0$, this equation determines its variation on the  timescale $\tau$ which is instantaneous relative to the timescale $t$. In the piecewise smooth dynamics literature this is sometimes referred to as providing the {\it blow-up} of the discontinuity surface $x_1=0$ into an interval $\lambda\in\sq{-1,+1}$. 

The system obtained by applying \eref{le} to \eref{2f}, using \eref{fnon} on the interval $\lambda\in\bb{\phi_{\eps*}^-,\phi_{\eps*}^+}$, is then 
\begin{equation}\label{lblowup}
(\e\dot\lambda,\dot x_2,\dot x_3)=\bb{f_1(x_1,x_2,x_3;\lambda),\;f_2(x_1,x_2,x_3;\lambda),\;f_3(x_1,x_2,x_3;\lambda)}\;.
\end{equation}
Where possible we can omit the arguments of $\e=\e(\lambda,\eps)$ without confusion. 
While standard geometrical singular perturbation theory does not apply to \eref{lblowup} because $\e$ is a function, we can show easily that this leads to the same critical manifold geometry as the conventional approach outlined in \sref{sec:blowup}. 

\begin{proposition}\label{Mtransfer}
The system \eref{lblowup} has equivalent slow-fast dynamics to the system \eref{blowup} on the discontinuity set $x_1=0$ in the critical limit $\e=0$.  
\end{proposition}

\begin{proof}
Rescaling time in \eref{lblowup} to $\tau=t/\e$, then setting $\e=0$ and $x_1=0$, gives the fast critical subsystem
\begin{equation}\label{lfast}
(\lambda',x_2',x_3')=\bb{f_1(0,x_2,x_3;\lambda),\;0,\;0}\;.
\end{equation}
The equilibria of this one-dimensional system form the manifold $$\op M^S=\cc{(\lambda,x_2,x_3)\in[-1,+1]\times\mathbb R^2:\;f_1(0,x_2,x_3;\lambda)=0}\;,$$
which is equivalent to the manifold \eref{Mdef}. 
This is an invariant manifold of the system \eref{lblowup} in the $\e=0$ limit everywhere that $\op M^S$ is normally hyperbolic, that is excepting the set
$$
\op L=\cc{(\lambda,x_2,x_3)\in\op M^S:\;\frac{\partial\;}{\partial \lambda}f_1(0,x_2,x_3;\lambda)=0}\;.
$$
Since $\frac{\partial\;}{\partial u}f_1=\phi_1'(u)\frac{\partial\;}{\partial\lambda}f_1$ and $\phi_1'(u)\neq0$ for $|u|<1$, 
this definition of $\op L$ is equivalent to \eref{Ldef}. 
Setting $\e=0$ and $x_1=0$ in \eref{lblowup} gives the slow critical subsystem
\begin{equation}\label{lblowup0}
(0,\dot x_2,\dot x_3)=\bb{f_1(0,x_2,x_3;\lambda),\;f_2(0,x_2,x_3;\lambda),\;f_3(0,x_2,x_3;\lambda)}\;,
\end{equation}
which defines dynamics in the {critical limit} $\e=0$ on $\op M^S$, which is exactly as given by \eref{blowup0}.  

\end{proof}

In \cite{j13error} an extension to the Filippov approach to piecewise smooth dynamics is therefore proposed, introducing directly a dummy timescale $\tau=t/\e$. Denoting the derivative with respect to this fast time $\tau$ we have simply
\begin{equation*}
\lambda'=\dot x_1={\bf f}({\bf x},\lambda)\cdot\nabla x_1\;,
\end{equation*}
and in full
\begin{equation}\label{dummy}
(\lambda',\dot x_2,\dot x_3)=\bb{f_1(x_1,x_2,x_3;\lambda),\;f_2(x_1,x_2,x_3;\lambda),\;f_3(x_1,x_2,x_3;\lambda)}\;.
\end{equation}
This permits study of the slow-fast critical dynamics without reference to regularization functions, providing a more direct analysis of sliding dynamics. 
The system \eref{dummy} directly permits us to consider the discontinuity surface $x_1=0$, and fixes the value that $\lambda$ takes on the interval $[-1,+1]$ inside the manifold $\op M^S$ (when it exists) for which the set $x_1=0$ is invariant, and specifying the variation $(\dot x_2,\dot x_3)$ on that manifold. In regions of $x_1=0$ where $\op M^S$ does not exist (where $f_1\neq0$ for $x_1=0$ and $\lambda\in[-1,+1]$), the fast subsystem of \eref{dummy} conveys the flow across the discontinuity from one region $x_1\gtrless0$ to the other. This analysis via the dummy system \eref{dummy} can be applied to reformulate the previous sections without regularization, yielding equivalent results, the main difference being that $\phi_1(u)$ and $\phi_1'(u)$ can be replaced everywhere by $\lambda$ and $1$ respectively.

\section{Nonlinear switching terms, and the ambiguity of regularization}\label{sec:prelude}

The vector field in \sref{sec:2fpert} demonstrates that qualitatively different geometry can be obtained if we consider perturbations outside the Sotomayor-Teixeira regularization. The dynamics that results is rather complex, however, so we shall briefly show an example of two non-equivalent smooth systems with the same piecewise smooth limit \eref{fns}. This highlights that there must be some ambiguity in how to regularize the discontinuous system into a smooth system. We show that the nonlinear dependence on $\lambda$ in \eref{fnon} provides a means to disctinguish between the different possible regularizations, and study them in the piecewise smooth limit. 

Consider two smooth system given by 
\begin{eqnarray}
(\dot x,\dot y)&=&(\phi_\eps(x),1)\;,\label{slin}\\
(\dot x,\dot y)&=&(\phi_\eps(x),1-2\phi_\eps^2(x))\;,\label{snon}
\end{eqnarray} 
where $\eps$ is a small positive parameter and $\phi_\eps(x)$ is a sigmoid function as defined in \eref{Peps}. We will refer to \eref{slin} and \eref{snon} as the linear and nonlinear systems, respectively. 

A common approach taken to studying sigmoid systems like \eref{slin} or \eref{snon} in applications, particularly when the systems above represent empirical models, is to study the system obtained by replacing $\phi_\eps(x)$ with ${\rm sign}(x)$ in the limit $\eps=0$. The result is a piecewise smooth system given for $x\neq0$ by 
\begin{equation}\label{sdisc}
(\dot x,\dot y)=({\rm sign}(x),-1)\;,
\end{equation}
for both \eref{slin} and \eref{snon}. 
The system is set-valued at $x=0$, and the question is then how to solve the differential inclusion at $x=0$. In the Filippov method we replace ${\rm sign}(x)$ with a switching parameter $\lambda$, such that $\lambda={\rm sign}(x)$ for $x\neq0$ and $\lambda\in[-1,+1]$ for $x=0$, and write 
\begin{equation}\label{sfil}
(\dot x,\dot y)=(\lambda,-1)\;,\qquad\lambda\in[-1,+1]\;.
\end{equation}
This is simply the system obtained by applying the convex combination \eref{ffil} to \eref{sdisc}. 
A \fref{fig:eg} illustrates, the two vector fields point towards $x=0$ so the entire discontinuity line must be a sliding region. 
The Sotomayor-Teixeira regularization smooths this by replacing $\lambda$ again with another sigmoid function $\phi_\eps(x)$ (obeying the same properties as in \eref{Peps})
$$(\dot x,\dot y)=\bb{\phi_\eps(x),-1}\;.$$

This recovers the system \eref{slin}, but not \eref{snon}. The dynamics of the two is crucially different. The linear system \eref{slin} has a critical manifold, given by $u=0$ if we assume $\phi_1(0)=0$. The dynamics on the manifold manifold is given by $(\dot u,\dot y)=(0,-1)$. 
To see this, introducing a fast coordinate $u=\eps x$, we have 
$$(\eps\dot u,\dot y)=(\phi_1(u),-1)\;$$
and apply standard geometrical singular perturbation \cite{f79,j95} in the limit $\eps\rightarrow0$. For small positive $\eps$ the dynamics near $u=0$ should be a small perturbation of $(\dot u,\dot y)=(0,-1)$. This is equivalent to the Filippov sliding dynamics on $x=0$ in the piecewise smooth system, $(\dot x,\dot y)=(0,-1)$. 

We failed to obtain \eref{snon} under regularization because we omitted what happens to the $\phi_\eps^2(x)$ term in the piecewise smooth limit (as $\eps\rightarrow0$). If we instead apply \eref{fnon} to \eref{sdisc} with ${\bf g}=(0,2)$, we obtain 
\begin{equation}\label{snonfil}
(\dot x.\dot y)=(\lambda,1-2\lambda^2)\;.
\end{equation}
When we regularize this we obtain 
$$(\dot x.\dot y)=(\phi_\eps(x),1-2\phi_\eps^2(x))\;,$$
and thus we have regained the nonlinear system \eref{snon} under regularization. The nonlinear system again has a critical manifold, given by $u=0$ if we assume $\phi_1(0)=0$, but now the dynamics on the manifold manifold is given by $(\dot u,\dot y)=(0,+1)$. 
This is seen by introducing a fast coordinate $u=\eps x$, giving 
$$(\eps\dot u,\dot y)=(\phi_1(u),1-2\phi_1^2(u))\;$$
and applying standard geometrical singular perturbation theory in the limit $\eps\rightarrow0$. For small positive $\eps$ the dynamics near $u=0$ should be a small perturbation of $(\dot u,\dot y)=(0,+1)$, and not in any way close to that of the linear system (which points in the opposite direction), as illustrated in \fref{fig:eg}. This dynamics is also not equivalent to the Filippov sliding dynamics of \eref{sdisc} via \eref{sfil}, but is instead equivalent to the nonlinear or `hidden' sliding dynamics of \eref{sdisc} via \eref{snonfil}, as described in \cite{j13error}. 

\begin{figure}[h!]\centering\includegraphics[width=0.95\textwidth]{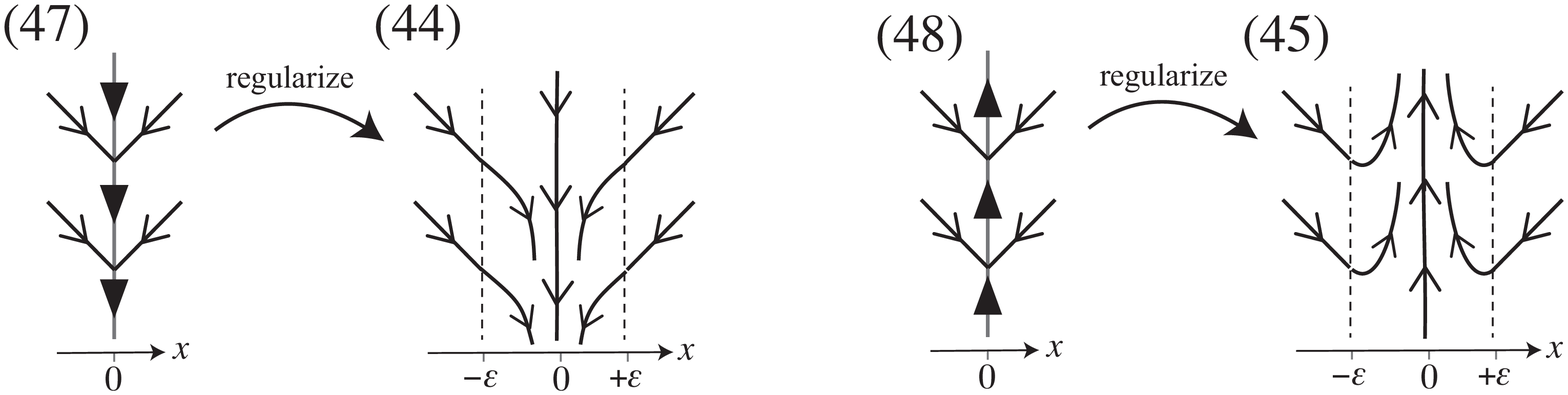}
\vspace{-0.3cm}\caption{\small\sf Two approaches to the piecewise smooth system \eref{sdisc}: the linear combination \eref{sfil} and nonlinear combination \eref{snonfil}, with their regularizations \eref{slin} and \eref{snon}. }\label{fig:eg}\end{figure}
%

\section{Closing Remarks}\label{sec:conc}

Regularizing or smoothing a discontinuity of course raises issues of uniqueness, namely that infinitely many qualitatively different smooth systems can have the same piecewise smooth limit. The idea of {\it nonlinear switching} terms --- nonlinear dependence on the parameter $\lambda$ in \eref{fnon} --- is that they provide a way of restoring uniqueness by distinguishing the limits of different smooth systems. We have shown here that the smooth system obtained from a two-fold singularity is structurally unstable if it depends only linearly on $\lambda$, but that a small perturbation, by terms that are nonlinear in $\lambda$, restores structural stability and allows transformation into the general local singularity expected in a smooth system, namely the folded singularity. 

For such a simple system, even taking its piecewise linear local normal form, the two-fold exhibits intricate and varied dynamics. Just how intricate becomes even more clear as we attempt to regularize the discontinuity, and study how the two-fold related to slow-fast dynamics of smooth systems. As well as insight into the dynamics that is seen upon simulated such a system, 
this adds a new facet to the question of the structural stability of the two-fold, which has remained a stimulating question since \cite{t90}.

An in-depth description of the dynamics that ensues in the different cases of two-folds, and the smoothings subject to perturbations, would be lengthy, and deserves future study elsewhere. 
As a demonstration we conclude with three examples showing the complex oscillatory attractors that can be formed by two-fold singularities. We take
\begin{enumerate}
\item[(i)] ${\bf f}^+=(-x_2,\sfrac25x_1+\sfrac1{10}x_2-1,\sfrac3{10}x_2-\sfrac15x_2x_3-\sfrac25)$, ${\bf f}^-=(x_3,\sfrac15x_2x_3-\sfrac35,\sfrac25x_3-1-x_1)$;
\item[(ii)] ${\bf f}^+=(-x_2,1+x_1,-\sfrac75)$, ${\bf f}^-=(x_3,-\sfrac9{10},1-\sfrac35x_1)$;
\item[(iii)] ${\bf f}^+=(-x_2+\sfrac1{10}x_1,x_1-\sfrac65,x_1-2)$, ${\bf f}^-=(x_3+\sfrac1{10}x_1,x_1+\sfrac{23}{100},1-x_1)$;
\end{enumerate}
and use \ref{fnon} with $\alpha=1/5$, and simulate them in \fref{fig:2fs}. 
\begin{figure}[h!]\centering\includegraphics[width=\textwidth]{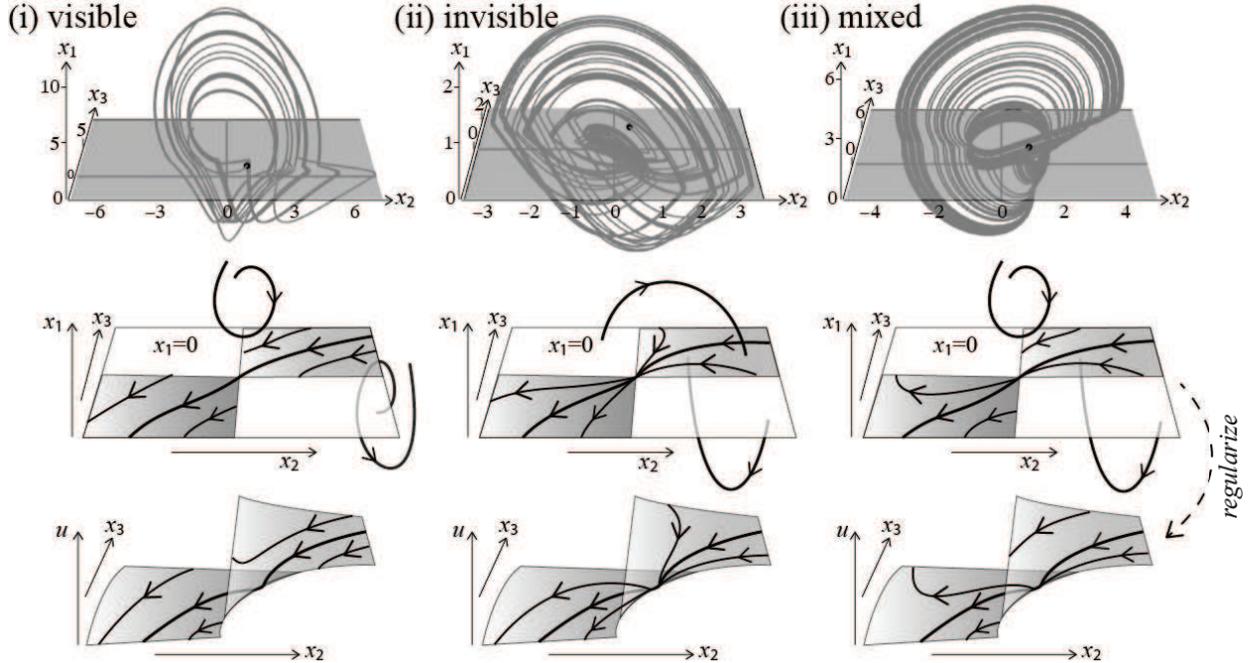}
\vspace{-0.3cm}\caption{\small\sf Three examples of attractor organised around a two-fold singularity. Showing: (i) a simulated trajectory, (ii) a sketch of the piecewise smooth flow inside and outside $x_1=0$ that gives rise to it, and (iii) the blow up on $x_1=0$.}\label{fig:2fs}\end{figure}

The numerical solutions apply Mathematica's {\sf NDSolve} to a regularized system, replacing $\lambda={\rm sign}\;x_1$ by a sigmoid function $\tanh(x_1/\eps)$ with $\eps=10^{-5}$. 
Further simulations not shown here verify that such dynamics persists with different monotonic smoothing functions, such as $x_1/\sqrt{\eps^2+x_1^2}$, and for different values of $\eps$. The first row shows the simulation of a single trajectory for a time interval $t=1000$ in $(x_1,x_2,x_3)$ space, as the flow attempts to switch between the the $x_1>0$ and $x_1<0$ flows, and the critical slow manifold flow. The piecewise smooth system is sketched in the second row, and the regularization in the third row, including the critical manifold $\op M^S$. 
The sliding vector field has canard trajectories passing from the righthand attracting branch to the lefthand repelling branch via the folded singularity; at a visible two-fold only one canard exists, at an invisible two-fold every sliding trajectory is a canard, and at a mixed two-fold a region of trajectories are canards.

The degeneracy in \sref{sec:2fdummy} gives some insight into why simulations of systems containing two-fold singularities with determinacy-breaking singularities exhibit highly sensitive dynamics, by relating it to a canard generating singularity in a singularly perturbed system. 


In the ongoing saga of the two-fold, the system \eref{2Fp} now succeeds \eref{2fg0} as our prototype for the local dynamics. The question of whether this constitutes a `normal form' has issues both in the piecewise smooth and slow-fast settings, but it is clear that \eref{2Fp} is structurally stable, and represents all classes of behaviour that occur both in the piecewise smooth system, and in its blow up to a slow-fast system. 



\bibliography{../grazcat}
\bibliographystyle{plain}

\clearpage

\end{document}